\documentclass[]{interact}

\usepackage{epstopdf}
\usepackage[caption=false]{subfig}

\usepackage[numbers,sort&compress]{natbib}
\bibpunct[, ]{[}{]}{,}{n}{,}{,}
\makeatletter
\def\NAT@def@citea{\def\@citea{\NAT@separator}}
\makeatother

\usepackage{amsthm}
\usepackage[usenames]{color}
\usepackage{xspace,colortbl}
\usepackage{latexsym,url,amscd,amssymb}
\usepackage{amsmath}
\usepackage{hyperref}
\usepackage{cases}
\usepackage{caption}
\usepackage{algorithm}
\usepackage{algpseudocode}
\usepackage{listings}
\usepackage{rotating}
\usepackage{framed}
\usepackage{amssymb}
\usepackage{amsmath}\numberwithin{equation}{section}

\DeclareMathOperator*{\argmin}{arg\,min}

\def\A{{\mathcal A}}
\def\B{{\mathcal B}}

\def\G{{\mathcal G}}
\def\H{{\mathcal H}}
\def\I{{\mathcal I}}

\def\M{{\mathcal M}}

\def\S{{\mathcal S}}
\def\T{{\mathcal T}}
\def\U{{\mathcal U}}

\def\X{{\mathcal X}}
\def\Y{{\mathcal Y}}
\def\Z{{\mathcal Z}}

\newtheorem{lemma}{Lemma}[section]
\newtheorem{definition}{Definition}[section]

\newtheorem{theorem}{Theorem}[section]
\newtheorem{assumption}{Assumption}[section]
\newtheorem{corollary}{Corollary}

\renewcommand{\thesection}{\arabic{section}}
\renewcommand{\thesubsection}{\arabic{section}.\arabic{subsection}}

\begin{document}

\title{On the Linear Convergence Rate of Generalized ADMM for Convex Composite Programming}

\author{
\name{Han Wang\textsuperscript{a}, Peili Li\textsuperscript{a,b} and Yunhai Xiao\textsuperscript{b}\thanks{CONTACT Yunhai Xiao. Email: yhxiao@henu.edu.cn}}
\affil{\textsuperscript{a}School of Mathematics and Statistics, Henan University, Kaifeng 475000, P.R. China; \\ \textsuperscript{b}Center for Applied Mathematics of Henan Province, Henan University, Zhengzhou 450046, P.R. China}
}

\maketitle

\begin{abstract}
Over the fast few years, the numerical success of the generalized alternating direction method of multipliers (GADMM) proposed by Eckstein \& Bertsekas [Math. Prog., 1992] has inspired intensive attention in analyzing its theoretical convergence properties.
In this paper, we devote to establishing the linear convergence rate of the semi-proximal GADMM (sPGADMM) for solving linearly constrained convex composite optimization problems. The semi-proximal terms contained in each subproblem possess the abilities of handling with multi-block problems  efficiently.
We initially present some important inequalities for the sequence generated by the sPGADMM, and then establish the local linear convergence rate under the assumption of calmness. As a by-product, the global convergence property is also discussed.
\end{abstract}

\begin{keywords}
Alternating direction method of multipliers; semi-proximal terms; error bound condition; calmness; linear convergence rate
\end{keywords}

\section{Introduction}
Let $\mathcal{X}$, $\mathcal{Y}$ and $\mathcal{Z}$ be finite-dimensional real Euclidean spaces each equipped with an inner product $\left \langle \cdot,\cdot \right \rangle$ and its induced norm $\left \| \cdot \right \|$, in which, $\mathcal{Y}:=\mathcal{Y}_1 \times \cdots\times\mathcal{Y}_m$ and $\mathcal{Z}:=\mathcal{Z}_1\times \cdots\times\mathcal{Z}_n$ are the Cartesian product of some finite-dimensional real Euclidean spaces.
We consider the following linearly constrained convex composite optimization problem
\begin{equation}\label{eq:9}
\begin{aligned}
\min&{\;\:f_1(y_1)+f_2(y_1,y_2,\ldots,y_m)+g_1(z_1)+g_2(z_1,z_2,\ldots,z_n)}\\
\textrm{s.t.}\:&\;\: \mathcal{A}_1^*y_1+\mathcal{A}_2^*y_2+\ldots+\mathcal{A}_m^*y_m+\mathcal{B}_1^*z_1+\mathcal{B}_2^*z_2+\ldots+\mathcal{B}_n^*z_n=c,
\end{aligned}
\end{equation}
where $f_1:\mathcal{Y}_1\rightarrow \left ( -\infty ,+\infty  \right ]$ and $g_1:\mathcal{Z}_1\rightarrow \left ( -\infty ,+\infty  \right ]$ are simple closed proper convex (not necessarily smooth) functions; $f_2:\mathcal{Y}\rightarrow \left ( -\infty ,+\infty  \right )$ and $g_2:\mathcal{Z}\rightarrow \left ( -\infty ,+\infty  \right )$ are  continuously differentiable convex quadratic functions;  $\mathcal{A}_i^*:\mathcal{Y}_i\rightarrow \mathcal{X}$  and $\mathcal{B}_i^*:\mathcal{Z}_i\rightarrow \mathcal{X}$ are the adjoints of the linear operators  $\mathcal{A}_i:\mathcal{X}\rightarrow \mathcal{Y}_i$  and $\mathcal{B}_i:\mathcal{X}\rightarrow \mathcal{Z}_i$, respectively; $c\in \mathcal{X}$ is a given data.
For convenience, we denote  the linear maps $\mathcal{A}:\mathcal{X}\rightarrow \mathcal{Y}$ and $\mathcal{B}:\mathcal{X}\rightarrow \mathcal{Z}$ such that their adjoint maps are given by
\[\mathcal{A}^*y=\sum_{i=1}^{m}\mathcal{A}_i^*y_i,\quad\forall y\in\mathcal{Y},\quad\mathcal{B}^*z=\sum_{j=1}^{n}\mathcal{B}_i^*z_i,\quad\forall z\in\mathcal{Z}.\]
Accordingly, the model (\ref{eq:9}) could be transformed into the following form
\begin{equation}\label{eq:1}
	\begin{aligned}
		\min&{\;\:f(y)+g(z)}\\
		\textrm{s.t.}\:&\;\: \mathcal{A}^*y+\mathcal{B}^*z=c,
	\end{aligned}
\end{equation}
where $f(y):=f_1(y_1)+f_2(y_1,y_2,\ldots,y_m)$ and $g(z):=g_1(z_1)+g_2(z_1,z_2,\ldots,z_n)$ are closed proper convex functions in the form of `nonsmooth+quadratic'.

Given its structure, problem (\ref{eq:1}) covers a wide range of
apparently related formulations in different scientific fields, including deep learning, compressive sensing, sparse coefficient estimation, etc.
For solving (\ref{eq:1}), a frequently used benchmark is the alternating direction method of multipliers (ADMM) which was originally proposed by Glowinski \& Marroco \cite{glowinski1975approximation} and Gabay \& Mercier \cite{Gabay1976ADA}.
Starting from $(x^0,y^0,z^0)\in \X \times \Y \times\Z$, the  iterative scheme of ADMM for solving (\ref{eq:1}) takes the following form
\begin{equation}\label{eq:3c}
\left\{
\begin{array}{l}
	y^{k+1} =   \argmin\limits_{y\in\Y} \ f(y)-\langle \A x^k, y\rangle+\dfrac{\sigma }{2} \| \mathcal{A}^*y+\mathcal{B}^*z^k-c\|^2,\\ [5mm]
	z^{k+1}  =  \argmin\limits_{z\in\Z}g(z)-\langle\B x^k,z\rangle+\dfrac{\sigma }{2} \| \mathcal{A}^*y^{k+1}+\mathcal{B}^*z-c\|^2,\\ [5mm]
	x^{k+1}  = x^k-\tau \sigma (\mathcal{A}^*y^{k+1}+\mathcal{B}^*z^{k+1}-c),
\end{array}
\right.
\end{equation}
where $x\in\X$ is a multiplier, $\sigma>0$ is a penalty parameter, and $\tau$ is a steplenght within the interval $(0, (1+\sqrt{5})/2)$. For the historical developments and applications of ADMM, one may refer to the review papers \cite{Boyd2011DistributedOA,eckstein2012augmented,Glowinski2014OnAD,Han2022ASO}. Based on the good performance of ADMM, many variants are constructed. For example, Yang et al. \cite{Yang2017AHA} proposed a homotopy-based ADMM to solve the linearly constrained separable convex minimization problems. Wu et al. \cite{Wu2017LinearizedBA} used linearization technique to the subproblems of block-wise ADMM, and obtained three kinds of linearized block-wise ADMM. Incidentally, the iteration scheme (\ref{eq:3c}) is not always well-defined, one can consult a counter-example in \cite{chen2017note} for more details.

Interestingly,  Gabay \cite{gabay1983chapter} showed that the type of ADMM (\ref{eq:3c}) with a unit steplength is equivalent to the Douglas-Rachford splitting (DRs) method \cite{douglas1956numerical,Lions1979SplittingAF} for the sum of two maximal monotone operators.
And then, Eckstein \& Bertsekas \cite{Eckstein1992OnTD} showed that  the DRs is actually an application of the proximal point algorithm (PPA) \cite{Rockafellar1976AugmentedLA,Rockafellar1976MonotoneOA}. In light of this, Eckstein \& Bertsekas \cite{Eckstein1992OnTD} considered a variant of PPA and then applied it on the minimization problem (\ref{eq:1}) which leads to the following  generalized version of ADMM (GADMM):
\begin{equation}\label{eq:20}
	\left\{
	\begin{array}{l}
		y^{k+1}=\argmin\limits_{y\in\Y} \ f(y)-\langle \A x^k, y\rangle+\dfrac{\sigma }{2} \| \mathcal{A}^*y+\mathcal{B}^*z^k-c\|^2,\\ [5mm]
		z^{k+1}= \argmin\limits_{z\in\Z} \ g(z)-\left \langle \mathcal{B}x^k,z\right \rangle+\dfrac{\sigma }{2}\left \| \rho (\mathcal{A}^*y^{k+1}+\mathcal{B}^*z^k-c )+\mathcal{B}^*(z-z^k)\right \|^2,\\ [5mm]
		x^{k+1}=x^k-\sigma \left (\rho (\mathcal{A}^*y^{k+1}+\mathcal{B}^*z^{k}-c)+\mathcal{B}^*(z^{k+1}-z^k)  \right ),
	\end{array}
	\right.
\end{equation}
where $\rho\in\left( 0,2\right) $ is a relaxation factor.
It is quite clear that the GADMM (\ref{eq:20}) with a unit relaxation factor is just the case of ADMM  (\ref{eq:3c}) with $\tau=1$. As pointed out by Adona et al. \cite{adona2019iteration} that, the GADMM is actually an instance of a hybrid proximal
extra-gradient framework of  Monteiro \& Svaiter \cite{monteiro2010complexity}.

Comparing with (\ref{eq:3c}) and (\ref{eq:20}), we see that GADMM still retains the benefits of ADMM in treating $f$ and $g$ separately, but the relaxation factor may make this method more flexible and efficient. Some earlier numerical studies of GADMM  in this shell can be found in \cite{bertsekas1982constrained,Eckstein1994ParallelAD}.
Eckstien \cite{Eckstein1994SomeSS} suggested adding proximal terms to the subproblems contained in (\ref{eq:3c}) to make them easier to solve.
Then, Fazel et al. \cite{Fazel2013HankelMR} proved that these proximal terms can be relaxed to positive and semi-definite operators (named semi-proximal terms).
It should be emphasized that the semi-proximal terms allow the subproblems in (\ref{eq:3c}) to be decomposed into some smaller ones and then solved individually. For example, note $f(y)=f_1(y_1)+f_2(y_1,y_2,\ldots,y_m)$, then the $y$-subproblem can be attained via the following process
\begin{equation}\label{sgsi}
y_m\rightarrow y_{m-1}\rightarrow\cdots\rightarrow y_2\rightarrow y_1\rightarrow y_2 \cdots\rightarrow y_{m-1}\rightarrow y_m
\end{equation}
by fixing other blocks with their latest values, which means that the (\ref{eq:3c}) has the advantage to resolve the potentially nonsolvability issue of the subproblems.
This updating order is called symmetric Gauss-Seidel (sGS) iteration which was firstly appeared in \cite{sun2015convergent} and later was analyzed by Li et al. \cite{li2016schur,li2019block}.
For more details on the developments and applications of sGS, one may refer to the Ph.D. thesis \cite{li2015two} and the papers of \cite{li2016majorized,chen2017efficient,li2018qsdpnal,xiao2018generalized}.
From \cite[Theorem 1]{li2019block}, we know that this updating order (\ref{sgsi}) is actually equivalent to adding a special semi-proximal term, and hence,
the convergence can be easily obtained from  Fazel et al. \cite[Theorem B.1]{Fazel2013HankelMR}.

Inspired by the success of sGS technique, it is natural to add a pair of semi-proximal terms to the subproblems in (\ref{eq:20}) so that GADMM
possesses the abilities of handling
multi-block problems in the form of (\ref{eq:9}). Actually, this idea was firstly attempted by Fang et al. \cite{fang2015generalized} in the sense of adding a proximal term to the $z$-subproblem in (\ref{eq:20}), but the $y$-subproblem was ignored.
This defect was remedied immediately by Adona et al. \cite{adona2019iteration}.
Let $\mathcal{S:Y\rightarrow Y}$ and $\mathcal{T:Z\rightarrow Z}\;$ be a pair of self-adjoint positive semidefinite (not necessarily positive definite) linear operators, the generalized ADMM with semi-proximal terms (abbr. sPGADMM) for solving the convex composite  programming (\ref{eq:1}) can be described as follows:
\newpage
\begin{framed}
\noindent
{\bf Algorithm sPGADMM:}
\vskip 1.0mm \hrule \vskip 1mm
\noindent
{\bf Step 0}: Let $\sigma \in(0,+\infty)$ and $\rho \in (0,2)$ be given parameters. Choose $(y^0,z^0,x^0)\in\textrm{dom}\:(f)\times \textrm{dom}\:(g)\times \mathcal{X} $. For $k=0,1,2,\ldots$, do the following steps iteratively:\\
{\bf Step 1}: Compute
		\begin{subequations}
			\begin{numcases}{}
				y^{k+1} =   \underset{y\in\Y}{\arg\min}\;f(y)-\langle \A x^k, y\rangle+\dfrac{\sigma }{2} \| \mathcal{A}^*y+\mathcal{B}^*z^k-c\|^2+\frac{1}{2}\left \| y-y^k \right \|_\S^2,\label{eq:a}\\
				z^{k+1}  =  \underset{z\in\Z}{\arg\min}\;g(z)-\left\langle \mathcal{B}x^k,z\right\rangle\notag \\ \qquad\qquad +\frac{\sigma }{2}\left \| \rho (\mathcal{A}^*y^{k+1}+\mathcal{B}^*z^k-c )+\mathcal{B}^*(z-z^k)\right \|^2+\frac{1}{2}\left \| z-z^k \right \|_\T^2,\label{eq:b}\\
				x^{k+1}  := x^k-\sigma \Big(\rho (\mathcal{A}^*y^{k+1}+\mathcal{B}^*z^{k}-c)+\mathcal{B}^*(z^{k+1}-z^k)\Big). \label{eq:c}
			\end{numcases}
		\end{subequations}
{\bf Step 2}: If a termination criterion is not met, set $k:=k+1$ and go to {\bf Step 1}.
\end{framed}

In this paper, we only focus on the linear convergence rate of the sPGADMM, because it plays a pivotal role to measure an algorithm's behavior. We note that there is a smaller number of  papers being devoted to  the linear convergence rate of the GADMM (\ref{eq:20}) and its related variants.
For example,
Fang et al. \cite{fang2015generalized} proved the iteration complexity in both the ergodic and a nonergodic senses for a special case of sPGADMM, and then, the ergodic iteration complexity result was further improved by Adona et al. \cite{adona2019iteration}. Corman \& Yuan \cite{corman2014generalized} and Tao \& Yuan \cite{tao2018optimal} analyzed the linear convergence rate of GADMM (\ref{eq:20}) from the aspect of generalized PPA, which is only a particular instance of sPGADMM.
Peng \& Zhang \cite{penglinear} established the linear convergence rate of GADMM (\ref{eq:20}) for the sequence $\{(z^k,x^k)\}$ instead of the whole sequence $\{(y^k,z^k,x^k)\}$ itself.
In this shell, Peng et al. \cite{Peng2022TheLC} proved the linear convergence rate of sPGADMM under certain settings, however, both $\S$ and $\T$ must be positive definite operators.
Our focus of this paper is on the sPGADMM scheme (\ref{eq:a})-(\ref{eq:c}) with $\S$ and $\T$ being semi-positive definite, which covers almost all the works aforementioned. We attempt to  establish the linear rate of
sPGADMM under a calmness condition, which holds automatically for convex composite piecewise linear-quadratic programming.
Here, it is worth emphasizing that our convergence analysis is inspired by Han et al. \cite{han2018linear}  who  provided a general linear rate convergence analysis for the semi-proximal ADMM of Fazel et al. \cite{Fazel2013HankelMR}. Nevertheless, we show that the proof process is not trivial, because
the relaxation factor brings many technical difficulties.

The remaining parts of this paper are organized as follows. Section \ref{sec2} is divided into two parts: Subsection \ref{sec2.1} presents some preliminary results on the optimimality conditions for problem (\ref{eq:1}), and Subsection \ref{sec2.2} briefly reviews the concepts of locally upper Lipschitz continuity and the calmness for multi-valued mappings. In Section \ref{sec3}, we provide a particularly useful inequality for the iteration sequence generated by sPGADMM, which is the key to the global convergence of sPGADMM. Section \ref{sec4} is devoted to building up a general local linear convergence rate under an error bound condition, and then obtaining a global linear convergence rate. Finally, we conclude our paper with some remarks in Section \ref{sec5}.

\section{Preliminaries}\label{sec2}
In this section, we  introduce some notations in the context and summarize some useful preliminaries for later analysis.

\subsection{Optimality conditions}\label{sec2.1}
For any two vectors $x\in \mathbb{R}^n$\:and\:$y\in \mathbb{R}^m$, we use $(x,y)$ to denote their adjunction, i.e., $(x,y)=(x^\top,y^\top)^\top$. For $p\geq 1$, we use $ \| x  \|_p$ to denote an $\ell_p$-norm of a vector $x$, and  $ \| x \|$ is short for $\ell_2$-norm. Let $\mathcal{G}$ be a positive definite matrix, the $\mathcal{G}$-norm of $x$ is denoted by $\|x\|_\mathcal{G}:=\sqrt{x^\top\mathcal{G}x}$. Given a closed convex set $C$, the distance of $x$ to $C$ regarding $\G$-norm is denoted as $\text{dist}_\mathcal{G}(x,C):=\textrm{inf}_{y\in C}\| x-y\|_\mathcal{G}$.

Recall that $f$ and $g$ in (\ref{eq:1}) are closed proper convex functions, we denote their subdifferential mappings  by $\partial f$ and $\partial g$, respectively. Moreover,  the subdifferential mappings of the closed proper convex functions are maximal monotone \cite[Theorem 12.17]{Rockafellar1998VariationalA}, i.e., there exist two self-adjoint and positive semidefinite operators $\Sigma _f$ and $\Sigma _g$ such that for all $y, y'\in \textrm{dom} (f)$,  $\xi \in \partial f(y) $ and $ {\xi }'\in \partial f({y}')$,
\begin{equation}\label{eq:y}
	 \langle {\xi }'-\xi, {y}'-y\rangle\geq  \|  {y}'-y\|_{\Sigma _f}^2,
\end{equation}
and for all $z, z'\in \textrm{dom} (g)$,  $\zeta  \in \partial g(z) $ and ${\zeta  }'\in \partial g({z}')$,
\begin{equation}\label{eq:z}
  \langle  {\zeta  }'-\zeta , {z}'-z\rangle  \geq \|  {z}'-z \|_{\Sigma _g}^2.
\end{equation}
It follows from \cite[Corollaries 28.2.2 and 28.3.1]{rockafellar1970convex} that, $( \bar{y},\bar{z})\in \textrm{ri}(\textrm{dom}(f)\times \textrm{dom}(g))$ is an optimal solution to  problem (\ref{eq:1}) if and only if there exists a Lagrange multiplier  $\bar{x}\in \mathcal{X}$ such that $\left ( \bar{y},\bar{z},\bar{x}\right )$ satisfies the following Karush-Kuhn-Tucker (KKT) system:
\begin{equation}\label{eq:2}
	\mathcal{A}\bar{x}\in \partial f(\bar{y}),\; \; \mathcal{B}\bar{z}\in \partial g(\bar{z})\; \; \textrm{and}\; \; \mathcal{A}^*\bar{y}+\mathcal{B}^*\bar{z}-c=0.
\end{equation}
The solution set to the KKT system (\ref{eq:2}) is denoted by $\bar{\Omega }$. The nonempty of $\bar{\Omega}$ can be guaranteed if a certain constraint qualification such as the Slater condition holds, that is, there exists $( y^0,z^0)\in \textrm{ri}(\textrm{dom}(f)\times \textrm{dom}(g))$ such that $\A^*y^0+\B^*z^0=c$,
where $\text{ri}(\cdot)$ denotes the relative interior of a given convex set. In this paper, instead of using an explicit constraint qualification, we only make the following assumption on the existence of a KKT point.
\begin{assumption}\label{assumption1}
	The KKT system (\ref{eq:2}) has a nonempty solution set, i.e., $\bar{\Omega}\neq\varnothing$.
\end{assumption}
For notational convenience, we let $u:=(y,z,x)\in\U$ with $\mathcal{U:=Y\times Z\times X}$ such that $y\in \mathcal{Y},z\in \mathcal{Z}$, and $x\in \mathcal{X}$. Define the KKT mapping $R:\mathcal{U}\rightarrow \mathcal{U}$ as
\begin{equation}\label{eq:22}
	R(u):=\begin{pmatrix}
		y-\textrm{Prox}_f(y+\mathcal{A}x)\\[2mm]
		z-\textrm{Prox}_g(z+\mathcal{B}x)\\[2mm]
		\mathcal{A}^*y+\mathcal{B}^*z-c
	\end{pmatrix},\qquad\forall u\in \mathcal{U},
\end{equation}
where $\textrm{Prox}_f\left ( \cdot \right )$ represents the proximal mapping of a closed proper convex function $f$. From optimization theory, it was known that the proximal mappings $\textrm{Prox}_f( \cdot)$ and $\textrm{Prox}_g ( \cdot )$ are globally Lipschitz continuous, which means that the mapping $R(\cdot)$ is continuous on $\mathcal{U}$ and $R(u)=0$ if and only if $u\in\bar{\Omega}$.
The KKT mapping of (\ref{eq:22}) originated from Han et al. \cite{han2018linear}, which plays a vital role in our subsequent analysis.

\subsection{Calmness}\label{sec2.2}
Let $F:\mathcal{X}\rightrightarrows \mathcal{Y}$ be a set-valued mapping with its graph denoted by gph$F$, and let $\mathbf{B}_\mathcal{Y}$ be a unit ball in $\mathcal{Y}$.
In the first place, we give the definition of locally upper Lipschitz continuity.

\begin{definition}[\cite{robinson1976implicit}]
	The multi-valued mapping $F:\mathcal{X}\rightrightarrows \mathcal{Y}$ is said to be Locally upper Lipschitz continuity at $x^0\in\mathcal{X}$ if there exist a constant $k_0>0$ along with a neighborhood $V$ of $x_0$ such that
	\[F\left ( x \right )\subseteq F\left ( x^0 \right )+k_0\left \| x-x^0 \right \|\mathbf{B}_\mathcal{Y},\; \; \; \; \forall x\in V.\]
\end{definition}
\noindent  It was known from Robinson \cite{robinson1981some} that, if a multivalued mapping $F:\mathcal{X}\rightrightarrows \mathcal{Y}$ is piecewise polyhedral, then $F$ is locally upper Lipschitz continuous at any $x^0\in\mathcal{X}$ with modulus $k_0$ independent of the choice of $x^0$.
A closed proper convex function $f:\mathcal{X}\rightarrow (-\infty ,+\infty ]$ is said to be piecewise linear-quadratic if dom$(f)$ is an union of finitely many polyhedral sets and on each of these polyhedral sets, $f$ is either an affine or a quadratic function. Sun \cite{sun1986monotropic} proved that, $f$ is piecewise linear-quadratic if and only if the graph of its subdifferential mapping $\partial f$ is piecewise polyhedral.
For more details, one may refer to Rockafellar \& Wets \cite[propositions 12.30 and 11.14]{Rockafellar1998VariationalA}

In the second place, we give the definition of calmness for  $F:\mathcal{X}\rightrightarrows \mathcal{Y}$ at $x^0$ for $y^0$ with $(x^0,y^0)\in\textrm{gph}F$. For more details on calmness, one can see  Dontchev \& Rockafellar \cite[Section 3.8 (3H)]{dontchev2009implicit} and Rockafellar \& Wets \cite{Rockafellar1998VariationalA} .
\begin{definition}[\cite{dontchev2009implicit,Rockafellar1998VariationalA}]
	Let $(x^0,y^0)\in\textrm{gph}F$. The multi-valued mapping $F:\mathcal{X}\rightrightarrows \mathcal{Y}$ is said to be calm at  $x^0$ for $y^0$ with modulus $k_0\geq  0$ if there exist a neighborhood $V$ of $x^0$ and a neighborhood $W$ of $y^0$ such that\[F\left ( x \right )\cap  W\subseteq F\left ( x^0 \right )+k_0\left \| x-x^0 \right \|\mathbf{B}_\mathcal{Y},\; \; \; \; \forall x\in V.\]
\end{definition}
\noindent  Therefore, we get that the subdifferential mapping  $\partial f(=:F)$ of a convex piecewise linear-quadratic function $f$ is calm at $x^0$ for $y^0$ meeting $(x^0,y^0)\in\textrm{gph}F$ with modulus $k_0\geq  0$ independent of the selection of $(x^0,y^0)$. Besides, it is also known from Dontchev \& Rockafellar \cite[Theorem 3H.3]{dontchev2009implicit} that, for any $(x^0,y^0)\in\textrm{gph}F$, the mapping $F$ is calm at $x^0$ for $y^0$ if and only if $F^{-1}$ (the inverse mapping of $F$) is metrically subregular at $y^0$ for $x^0$, i.e., there exist a constant $k'_0\geq  0$, a neighborhood $W$ of $y^0$, and a neighborhood $V$ of $x^0$ such that
\begin{equation}
	\textrm{dist}(y,F(x^0))\leq k'_0\textrm{dist}(x^0,F^{-1}(y)\cap V),\; \; \; \forall y\in W,
\end{equation}
where $\text{dist}(\cdot,\cdot)$ represents a distance regarding an appropriate norm.
\section{Global convergence of sPGADMM}\label{sec3}
In order to prove the linear convergence rate, we must establish the global convergence of sPGADMM because it is essential to our subsequent analysis. Certainly, the convergence can be followed directly from the framework of PPA, but there is no  proof for optimization problems especially including a pair of semi-definite operators $\S$ and $\T$ in the current literature.
In our analysis, the following elementary equality is used times and times again
\begin{equation}\label{eq:7}
	2\left \langle v_1,\mathcal{G}v_2\right \rangle=\left \| v_1 \right \|_\mathcal{G}^2+\left \|v_2\right \|_\mathcal{G}^2-\left \| v_1 -v_2\right \|_\mathcal{G}^2=\left \| v_1 +v_2\right \|_\mathcal{G}^2-\left \| v_1 \right \|_\mathcal{G}^2-\left \| v_2\right \|_\mathcal{G}^2,
\end{equation}
where $v_1$ and $v_2$ are vectors in the same finite dimensional real Euclidean space endowed with inner product $\left \langle \cdot,\cdot \right \rangle$ and induced norm $\left \| \cdot \right \|$, and $\mathcal{G}$ is an arbitrary self-adjoint positive semi-definite linear operator.

For notational convenience, we denote $y_e=y-\bar{y}$, $z_e=z-\bar{z}$, and $x_e=x-\bar{x}$.  Now, we are ready to give some inequalities used in the later analysis.

\begin{lemma}\label{lemma1}
	Let $\left \{ u^k\right \}$ be the sequence generated by the sPGADMM scheme \eqref{eq:a}-\eqref{eq:c}. For any $k\geq 0$, the following results hold:\\
	\begin{equation}\label{eq:4}	
			\langle \mathcal{B}^* ( z^{k+1}-z^k ),x^{k+1}-x^k  \rangle\geq \frac{1}{2} \Big( \|z^{k+1}-z^k\| _\T^2- \|z^{k}-z^{k-1}  \| _\T^2\Big),
	\end{equation}
and
\begin{equation}\label{eq:5}
	\begin{aligned}
		&\Big\langle x_e^{k+1}+\sigma ( 1-\rho)\mathcal{B}^*(z^{k+1}-z^k) ,\mathcal{A}^*y_e^{k+1} +\mathcal{B}^*z_e^{k+1} \Big\rangle+\frac{1}{2}\sigma \rho \| \mathcal{A}^*y_e^{k+1} +\mathcal{B}^*z_e^{k+1}  \|^2\\
	=	& ( 2\sigma \rho )^{-1} \bigg[ \| x_e^{k}+\sigma ( 1-\rho )\mathcal{B}^*z_e^{k}  \|^2 - \| x_e^{k+1}+\sigma( 1-\rho)\mathcal{B}^*z_e^{k+1} \|^2\bigg  ],		
	\end{aligned}
\end{equation}
and
\begin{equation}\label{eq:6}
	\begin{aligned}
		& \langle \mathcal{B}^*( z_e^{k+1}-z_e^k ),\mathcal{A}^*y_e^{k+1} \rangle - ( 2\rho )^{-1} ( 2-\rho) \| \mathcal{B}^* (z_e^{k+1}- z_e^k ) \|^2\\
		&\leq \frac{1}{2}\Big ( \| \mathcal{B}^*z_e^k  \|^2- \| \mathcal{B}^*z_e^{k+1}  \|^2\Big)+\frac{1}{2} ( \sigma \rho)^{-1}\Big ( \| z^k-z^{k-1} \|  _\T^2-\| z^{k+1}-z^k \|_\T^2\Big ).
	\end{aligned}
\end{equation}
\end{lemma}
\begin{proof}
	Firstly, invoking the first-order optimality condition for (\ref{eq:b}), we get
	\[\mathcal{B}\left [ x^k-\sigma \rho \left ( \mathcal{A}^* y^{k+1}+\mathcal{B}^*z^{k+1}-c\right )+\sigma \left ( 1-\rho  \right )\mathcal{B}^*(z^{k+1}-z^k) \right ]-\T\left ( z^{k+1}-z^k \right )\in \partial g(z^{k+1}),\]
	which, together with (\ref{eq:c}), implies
	\begin{equation}\label{eq:3}
		\mathcal{B}x^{k+1}-T(z^{k+1}-z^k)\in \partial g(z^{k+1}),
	\end{equation}
	and
	\[\mathcal{B}x^k-T(z^k-z^{k-1})\in \partial g(z^k).\]
	From the maximal monotonicity of $\partial g(\cdot)$, we have
	\begin{equation*}
		\left \langle \mathcal{B}^*\left ( z^{k+1}-z^k \right ),x^{k+1}-x^k \right \rangle-\left \langle \T\left ( z^{k+1}-z^k \right )-\T\left ( z^k-z^{k-1}  \right ),z^{k+1}-z^k  \right \rangle\geq 0.
	\end{equation*}
    The above inequality yields
   \begin{equation*}
   	\begin{aligned}
   		\left \langle \mathcal{B}^*\left ( z^{k+1}-z^k \right ),x^{k+1}-x^k \right \rangle&\geq \left \langle \T\left ( z^{k+1}-z^k \right )-\T\left ( z^k-z^{k-1}  \right ),z^{k+1}-z^k  \right \rangle\\
   		&\geq \left \| z^{k+1}-z^k \right \|_\T^2-\frac{1}{2}\left ( \left \| z^{k+1}-z^k \right \|_\T^2+\left \| z^{k}-z^{k-1} \right \|_\T^2 \right )\\
   		&=\frac{1}{2}\left ( \left \| z^{k+1}-z^k \right \|_\T^2-\left \| z^{k}-z^{k-1} \right \|_\T^2 \right ),
   	\end{aligned}
   \end{equation*}
   which indicates that (\ref{eq:4}) is true.

   Secondly, it follows from (\ref{eq:c}) that
   \[x^k=x^{k+1}+\sigma \rho \left ( \mathcal{A}^*y^{k+1}+\mathcal{B}^*z^{k+1}-c \right )+\sigma \left ( 1-\rho  \right )\mathcal{B}^*(z^{k+1}-z^k).\]
   It is easy to see that
   \begin{equation}\label{eq:8}
   	\left [x_e^k+\sigma \left ( 1-\rho  \right )\mathcal{B}^*z_e^k  \right ]-\left [x_e^{k+1}+\sigma \left ( 1-\rho  \right )\mathcal{B}^*z_e^{k+1}  \right ]=\sigma \rho \left ( \mathcal{A}^*y_e^{k+1}+\mathcal{B}^*z_e^{k+1} \right ).
   \end{equation}
   Subsequently, using the relation (\ref{eq:8}) and the elementary equality (\ref{eq:7}), we know that the equality (\ref{eq:5}) holds.

   Finally, according to (\ref{eq:8}), we can deduce that
   \begin{equation*}
   	\begin{aligned}
   		\mathcal{A}^*y_e^{k+1}&=\left (\mathcal{A}^*y_e^{k+1}+\mathcal{B}^*z_e^{k+1}  \right )-\mathcal{B}^*z_e^{k+1}\\
   		&=\left ( \sigma \rho   \right )^{-1}\left [x_e^k-x_e^{k+1}+\sigma \left ( 1-\rho  \right )\left (\mathcal{B}^*z_e^k - \mathcal{B}^*z_e^{k+1}\right ) \right ]-\mathcal{B}^*z_e^{k+1}.
   	\end{aligned}
   \end{equation*}
   Combining this equality with (\ref{eq:7}), we  obtain that
   \begin{equation*}
   	\begin{split}
   		&\left \langle  \mathcal{B}^*\left (z_e^{k+1}- z_e^k \right ),\mathcal{A}^*y_e^{k+1}\right \rangle\\
   		=&\left ( \sigma \rho   \right )^{-1}\left \langle  \mathcal{B}^*\left (z_e^{k+1}- z_e^k \right ),x_e^k-x_e^{k+1}  \right \rangle-\left \langle \mathcal{B}^*\left (z_e^{k+1}- z_e^k \right ),\mathcal{B}^*z_e^{k+1} \right \rangle\\
   		&+\rho ^{-1} \left ( 1-\rho  \right )\left \langle  \mathcal{B}^*\left (z_e^{k+1}- z_e^k \right ),\mathcal{B}^*\left (z_e^k - z_e^{k+1}  \right ) \right \rangle\\
   		=&\left ( \sigma \rho   \right )^{-1}\left \langle  \mathcal{B}^*\left (z_e^{k+1}- z_e^k \right ),x_e^k-x_e^{k+1}  \right \rangle\\
   		& +\frac{1}{2}\left (\left \|\mathcal{B}^*z_e^k \right \|^2-\left \|\mathcal{B}^*z_e^{k+1}   \right \|^2  \right )-\left ( 2\rho  \right )^{-1}\left ( 2-\rho  \right )\left \| \mathcal{B}^*\left (z_e^{k+1}- z_e^k \right ) \right \|^2,
   	\end{split}
   \end{equation*}
   which, together with (\ref{eq:4}), implies the inequality (\ref{eq:6}).
\end{proof}

Utilizing Lemma \ref{lemma1}, we  can establish the difference of the distance to $\bar{\Omega }$ for two consecutive points of the sequence $\left \{ u^k\right \}$.
\begin{lemma}\label{lemma2}
	Let the sequence $\left \{ u^k\right \}$ be generated by sPGADMM  \eqref{eq:a}-\eqref{eq:c}. For any $k\geq 0$, define
	\begin{equation}\label{eq:10}
		\begin{cases}
			\phi _{k+1}:=\left ( \sigma \rho  \right )^{-1}\left \| x_e^{k+1}+\sigma \left ( 1-\rho  \right )\mathcal{B}^*z_e^{k+1} \right \|^2+\left \| y_e^{k+1} \right \|_S^2+\left \| z_e^{k+1} \right \|_\T^2\\[2mm]
			\qquad\qquad+\sigma \left ( 2-\rho  \right )\left \| \mathcal{B}^*z_e^ {k+1}\right \|^2+\left ( 2-\rho  \right )\rho ^{-1}\left \| z_e^{k+1}-z_e^k \right \|_\T^2,\\[2mm]
			t_{k+1}:=2\left \| y_e^{k+1} \right \|_{\Sigma _f}^2+2\left \| z_e^{k+1} \right \|_{\Sigma _g}^2+\left \| y^{k+1}-y^k \right \|_S^2+\left \| z^{k+1}-z^k \right \|_\T^2\\[2mm]
			\qquad\qquad+\sigma \left ( 2-\rho  \right )^2\rho ^{-1}\left \|\mathcal{B}^*\left ( z^{k+1}-z^k   \right )\right \|^2.
		\end{cases}
	\end{equation}
	Then, it holds that
	\begin{equation}\label{eq:11}
		\phi _k-\phi _{k+1}\geq t_{k+1}+(2-\rho )\sigma \left \| \mathcal{A}^*y_e^{k+1}+\mathcal{B}^*z_e^{k+1} \right \|^2.
	\end{equation}
\end{lemma}
\begin{proof}
	Notice that (\ref{eq:a}) can be rewritten as
	$$
	\mathcal{A}\left [ x^k-\sigma \left ( \mathcal{A}^* y^{k+1}+\mathcal{B}^*z^{k+1}-c\right )+\sigma \mathcal{B}^*\left ( z^{k+1} -z^k\right ) \right ]-S\left ( y^{k+1}-y^k \right )\in \partial f(y^{k+1}).
	$$
    It follows from (\ref{eq:c}) that
    \begin{equation}\label{eq:12}
    \begin{aligned}
    		\mathcal{A}&\left [ x^{k+1}  -\sigma \left ( 1-\rho  \right ) \left ( \mathcal{A}^*y^{k+1}+\mathcal{B}^*z^{k+1}-c \right ) +\sigma \left ( 2-\rho  \right )\mathcal{B}^*\left ( z^{k+1}-z^k \right )\right ]\\[2mm]
    &-S\left ( y^{k+1}-y^k \right )\in \partial f(y^{k+1}).
    \end{aligned}
    \end{equation}
   Combining (\ref{eq:y}) and (\ref{eq:z}) with (\ref{eq:3}) and (\ref{eq:12}), we have
   \begin{equation*}
   	\begin{aligned}
   		\left \langle x_e^{k+1},\mathcal{A}^*y_e^{k+1} \right \rangle &-\sigma \left ( 1-\rho  \right )\left \langle  \mathcal{A}^*y_e^{k+1}+\mathcal{B}^*z_e^{k+1} , \mathcal{A}^*y_e^{k+1}\right \rangle-\left \langle  S\left ( y_e^{k+1}-y_e^k \right ), y_e^{k+1}\right \rangle\\[2mm]
   		&+\sigma \left ( 2-\rho  \right ) \left \langle \mathcal{B}^*\left ( z_e^{k+1}-z_e^k \right ),\mathcal{A}^*y_e^{k+1}  \right \rangle \geq \left \| y_e^{k+1} \right \|_{\Sigma _f}^2,
   	\end{aligned}
   \end{equation*}
and
   $$
   \left \langle  x_e^{k+1},\mathcal{B}^*z_e^{k+1}\right \rangle-\left \langle  T\left ( z_e^{k+1}-z_e^k \right ), z_e^{k+1} \right \rangle\geq \left \| z_e^{k+1} \right \|_{\Sigma _g}^2.
   $$
   Adding both sides of the inequalities, we have
   \begin{equation}\label{eq:13}
   	\begin{aligned}
   		&\left \langle x_e^{k+1},\mathcal{A}^*y_e^{k+1} +\mathcal{B}^*z_e^{k+1} \right \rangle-\sigma \left ( 1-\rho  \right )\left \langle \mathcal{A}^*y_e^{k+1}+\mathcal{B}^*z_e^{k+1} , \mathcal{A}^*y_e^{k+1}\right \rangle\\[2mm]
   &+\sigma \left ( 2-\rho  \right )\left \langle \mathcal{B}^*\left ( z_e^{k+1}-z_e^k \right ) ,\mathcal{A}^*y_e^{k+1}\right \rangle
   		-\left \langle \S\left ( y_e^{k+1}-y_e^k \right ),y_e^{k+1} \right \rangle\\[2mm]
   &-\left \langle \T\left ( z_e^{k+1}-z_e^k \right ),z_e^{k+1} \right \rangle\\[2mm]
   & \geq \left \| y_e^{k+1} \right \|_{\Sigma _f}^2+\left \| z_e^{k+1} \right \|_{\Sigma _g}^2,
   	\end{aligned}
   \end{equation}
   which can be reformulated as
   \begin{equation}\label{eq:14}
   	\begin{split}
   		&\left \langle x_e^{k+1}+\sigma \left ( 1-\rho  \right )\mathcal{B}^*z_e^{k+1},\mathcal{A}^*y_e^{k+1} +\mathcal{B}^*z_e^{k+1} \right \rangle-\sigma \left ( 1-\rho  \right )\left \|\mathcal{A}^*y_e^{k+1} +\mathcal{B}^*z_e^{k+1} \right \|^2\\[2mm]
   		&+\sigma \left ( 2-\rho  \right )\left \langle \mathcal{B}^*\left ( z_e^{k+1}-z_e^k \right ) ,\mathcal{A}^*y_e^{k+1}\right \rangle-\left \langle \S\left ( y_e^{k+1}-y_e^k \right ),y_e^{k+1} \right \rangle\\[2mm]
   &-\left \langle \T\left ( z_e^{k+1}-z_e^k \right ),z_e^{k+1} \right \rangle\\[2mm]
   		&\geq \left \| y_e^{k+1} \right \|_{\Sigma _f}^2+\left \| z_e^{k+1} \right \|_{\Sigma _g}^2.
   	\end{split}
   \end{equation}
    Next, we analyze the left-hand-side of (\ref{eq:14}). Using the elementary equality (\ref{eq:7}), it is easy to see that
    \begin{equation}\label{eq:15}
    	\left \langle S\left (y_e^{k+1}-y_e^k \right ),y_e^{k+1} \right \rangle=\frac{1}{2}\left ( \left \| y_e^{k+1}- y_e^k\right \| _\S^2+\left \|  y_e^{k+1}\right \|_\S^2-\left \|  y_e^k\right \|_\S^2\right ),
    \end{equation}
    and
    \begin{equation}\label{eq:16}
    	\left \langle \T\left ( z_e^{k+1}-z_e^k \right ),z_e^{k+1} \right \rangle=\frac{1}{2}\left ( \left \| z_e^{k+1}- z_e^k\right \| _\T^2+\left \|  z_e^{k+1}\right \|_\T^2-\left \|  z_e^k\right \|_T^2\right ).
    \end{equation}
    Substituting (\ref{eq:5}), (\ref{eq:6}), (\ref{eq:15}) and (\ref{eq:16}) into (\ref{eq:14}), we have
    \begin{equation*}
    	\begin{split}
    		&\left ( \sigma \rho  \right )^{-1}\left \| x_e^{k}+\sigma \left ( 1-\rho  \right )\mathcal{B}^*z_e^{k} \right \|^2 +\left \|  y_e^k\right \|_\S^2+\left \|  z_e^k\right \|_\T^2+\sigma \left ( 2-\rho  \right )\left \|\mathcal{B}^*z_e^k \right \|^2 +\rho ^{-1}\left ( 2-\rho  \right )\left \| z^k- z^{k-1}  \right \|_\T^2\\[2mm]
    		&\geq \left ( \sigma \rho  \right )^{-1}\left \| x_e^{k+1}+\sigma \left ( 1-\rho  \right )\mathcal{B}^*z_e^{k+1} \right \|^2 +\left \|  y_e^{k+1}\right \|_S^2+\left \|  z_e^{k+1}\right \|_T^2+\sigma \left ( 2-\rho  \right )\left \|\mathcal{B}^*z_e^{k+1}   \right \|^2\\[2mm]
    		&+\rho ^{-1}\left ( 2-\rho  \right )\left \| z^{k+1}- z^k  \right \|_\T^2 +2\left \| y_e^{k+1} \right \|_{\Sigma _f}^2+2\left \| z_e^{k+1} \right \|_{\Sigma _g}^2+(2-\rho )\sigma \left \| \mathcal{A}^*y_e^{k+1}+\mathcal{B}^*z_e^{k+1} \right \|^2\\[2mm]
    		&+\left \| y^{k+1}-y^k \right \|_\S^2+\left \| z^{k+1}-z^k \right \|_\T^2+\sigma \left ( 2-\rho  \right )^2\rho ^{-1}\left \|\mathcal{B}^*\left ( z^{k+1}-z^k   \right )\right \|^2.
    	\end{split}
    \end{equation*}
    Therefore, the assertion (\ref{eq:11}) holds by recalling the definitions in (\ref{eq:10}).
\end{proof}

     The inequality (\ref{eq:11}) is essential to establish the global convergence of sPGADMM. Although the global convergence of sPGADMM has been well studied in the literature especially from the respect of generalized PPA framework, we present a detailed proof for the completeness of global convergence.
\begin{theorem}\label{theorem1}
	Suppose that Assumption \ref{assumption1} holds. Assume that $\Sigma _f+\S+\sigma \mathcal{A} \mathcal{A}^* $ and $\Sigma _g+\T+\sigma \mathcal{B}  \mathcal{B}^* $ are positive definite. Let the sequence $\left\{u^k\right\}$ be generated by sPGADMM scheme \eqref{eq:a}-\eqref{eq:c}. Then the sequence $\left\{(y^k,z^k)\right\}$ converges to an optimal solution of (\ref{eq:1}) and $\left\{x^k\right\}$ converges to an optimal solution of the corresponding dual problem.
\end{theorem}
\begin{proof}
	Notice that $\rho\in(0,2)$, we get from (\ref{eq:11}) that $\{\phi_k\}$ is a nonnegative and monotonically nonincreasing sequence. Consequently, ${\phi_k}$ is  bounded, so do the following sequences:
	\begin{equation}\label{eq:17}
		\left \{  \left \| x_e^{k}+\sigma \left ( 1-\rho  \right )\mathcal{B}^*z_e^{k} \right \|\right \},\; \; \left \{ \left \|z_e^k \right \|_{\sigma\mathcal{B}\mathcal{B}^* } \right \},\; \; \left \{ \left \| z^{k+1}- z^k\right \|_\T \right \},\; \; \left \{  \left \| z_e^k \right \|_\T\right \}\; \textrm{and} \;\left \{ \left \| y_e^k \right \|_\S \right \}.
	\end{equation}
    Besides, from the inequality (\ref{eq:11}), it is clear that
    \begin{equation}\label{eq:18}
    	\begin{aligned}
    		&\lim_{k\rightarrow \infty }\left \| \mathcal{A}^*y_e^{k+1}+\mathcal{B}^*z_e^{k+1} \right \|= 0,\; \; \lim_{k\rightarrow \infty }\left \| z^{k+1}-z^k\right \|_{\sigma\mathcal{B}^*\mathcal{B}}= 0,\; \; \lim_{k\rightarrow \infty }\left \| z^{k+1}-z^k\right \|_\T=0,\\
    		&\lim_{k\rightarrow 0}\left \| z_e^{k+1} \right \|_{\Sigma _g}= 0,\; \; \lim_{k\rightarrow 0}\left \| y^{k+1}-y^k \right \|_\S=0,\; \;\textrm{and}\;  \lim_{k\rightarrow 0}\left \| y_e^{k+1} \right \|_{\Sigma _f}= 0.
    	\end{aligned}
    \end{equation}
    Furthermore, by the use of
    \begin{equation}\label{eq:19}
    	\left \| \mathcal{A}^*y_e^{k+1} \right \|\leq \left \| \mathcal{A}^*y_e^{k+1}+\mathcal{B}^*z_e^{k+1} \right \|+\left \|  \mathcal{B}^*z_e^{k+1}\right \|,
    \end{equation}
 we know that $\left \{ \left \| y_e^{k+1} \right \| _{\sigma \mathcal{A}\mathcal{A}^* }\right \}$ is  bounded, and so do the sequence $\left \{ \left \| y_e^{k+1} \right \| _{\Sigma _f+\S+\sigma \mathcal{A}\mathcal{A}^* }\right \}$.  In light of the positive definiteness of $\Sigma _f+\S+\sigma \mathcal{A} \mathcal{A}^* $, the sequence $\left \{ \left \| y_e^k \right \| \right \}$ is also bounded.
Similarly, the sequences $\left\{\left \|z_e^k \right \|_{\sigma\mathcal{B}\mathcal{B}^* }\right\}$,  $\left\{\left \| z_e^k \right \|_\T\right\}$, and $\left\{\left \| z_e^k \right \|_{\Sigma _g}\right\}$ are all bounded. Then, the sequence $\left \{ \left \| z_e^k \right \| \right \}$ is also bounded because the operator $\Sigma _g+\T+\sigma \mathcal{B}\mathcal{B}^* $ is assumed to be positive definite. The boundedness of $\left \{  \left \| x_e^{k}+\sigma \left ( 1-\rho  \right )\mathcal{B}^*z_e^{k} \right \|\right \}$ and $\left \{ \left \| z_e^k \right \| \right \}$ further indicates that the sequence $\left \{  \left \| x_e^{k}\right \|\right \}$ is bounded. According to the above arguments, it yields that the sequence $\left \{  (y^k,z^k,x^k)\right \}$ is bounded.

From the boundedness of the sequence $\left \{  (y^k,z^k,x^k)\right \}$, we know that there exists at lease one subsequence which converges to a cluster point. Suppose that $\left \{ {(y^{k_i},z^{k_i},x^{k_i})}\right \}$ is a subsequence of ${(y^k,z^k,x^k)}$ converging to $(y^\infty ,z^\infty ,x^\infty )\in \mathcal{Y\times Z\times X}$.    Taking limits on both sides of (\ref{eq:3}) and (\ref{eq:12}) along the subsequence $\left \{ {(y^{k_i},z^{k_i},x^{k_i})}\right \}$, using (\ref{eq:17}) and (\ref{eq:18}), and noticing the closedness of the graphs of $\partial f$ and$\partial g$ \cite[p.80]{borwein2006convex}, we can get that
$$
\mathcal{A}x^\infty\in \partial f(y^\infty),\qquad \mathcal{B}x^\infty\in \partial g(z^\infty),\qquad \mathcal{A}^*y^\infty+\mathcal{B}^*z^\infty-c=0,
$$
which implies that $(y^\infty ,z^\infty ,x^\infty )$ is a solution to the KKT system (\ref{eq:2}). Hence, $(y^\infty ,z^\infty )$ is an optimal solution to problem (\ref{eq:1}) and that $x^\infty$ is an optimal solution to the corresponding dual problem.

We show that $(y^\infty ,z^\infty ,x^\infty )$  is actually the unique limit of $\left \{ {(y^k,z^k,x^k)}\right \}$. One can replace $\left ( \bar{y},\bar{z},\bar{x} \right )$ with $\left ( y^\infty ,z^\infty ,x^\infty  \right )$ in  (\ref{eq:2}) because $\left ( y^\infty ,z^\infty ,x^\infty  \right )$ is also satisfied from the above arguments. It is obvious that for $\rho\in(0,2)$ the subsequence $\left \{ \phi _{k_i} \right \}$ converges to 0 as $k_i\rightarrow \infty$. Because the subsequence comes from nonincreasing sequences, we can get that
    \begin{equation*}
    	\lim_{k\rightarrow 0}\phi _k=0.
    \end{equation*}
   Consequently, it holds that
   \[\lim_{k\rightarrow 0}\left (\left \| y_e^k \right \|_\S+\left \| y_e^{k+1} \right \|_{\Sigma _f}\right )+\left ( \left \|z_e^k \right \|_{\sigma\mathcal{B}\mathcal{B}^* }+\left \| z_e^k \right \|_\T+\left \| z_e^k \right \|_{\Sigma _g}   \right )=0.\]
   Thus, one can get $\lim_{k\rightarrow \infty }z^k=z^\infty $ as the operator $\Sigma _g+\T+\sigma\mathcal{B}\mathcal{B}^* $ is positive definite. This, combines with (\ref{eq:19}) implies that
   \[\lim_{k\rightarrow \infty }\left \| y_e^k  \right \|_{\sigma \mathcal{A}\mathcal{A}^* }+\left \| y_e^k \right \|_\S+\left \| y_e^{k+1} \right \|_{\Sigma _f}=0,\]
   which, from the fact that $\Sigma _f+\S+\sigma \mathcal{A} \mathcal{A}^* $ being positive definite, shows that $\lim_{k\rightarrow \infty }y^k=y^\infty $. Then, employing  (\ref{eq:17}), we acquire that $\lim_{k\rightarrow \infty }x^k=x^\infty$. Therefore, it gets that the whole sequence $\left \{ {(y^k,z^k,x^k)}\right \}$ converges to $(y^\infty ,z^\infty ,x^\infty )$ in the case of $\rho\in (0,2)$.
\end{proof}
\section{Linear convergence of sPGADMM}\label{sec4}
In this section, we are devoted to proving the local linear convergence rate of sPGADMM for convex composite optimization problem (\ref{eq:1}).
For this purpose, we list some notations which will be used in the following analysis. First, for any $\rho \in \left ( 0,2 \right )$, we define
\[l_\rho :=\left\{\begin{matrix}
	\dfrac{1}{\rho }, & \rho \in \left ( 0,\dfrac{1+\sqrt{5}}{2} \right ],\\
	\dfrac{2-\rho }{\rho -1},& \rho \in \left ( \dfrac{1+\sqrt{5}}{2},2\right ),
\end{matrix}\right.\qquad\qquad
h_\rho :=\left\{\begin{matrix}
	1-\min\left \{ \rho ,\rho ^{-1} \right \}, & \rho \in \left ( 0,\dfrac{1+\sqrt{5}}{2} \right ],\\
	2-\rho,\; \; \; \; \; \; \; \; \; \; \; \; \; \; \; \; \;   & \rho \in \left ( \dfrac{1+\sqrt{5}}{2},2\right ),
\end{matrix}\right.\]
which implies that
$$
l_\rho> 0, \ 0< \rho\, l_\rho < 2, \quad  \textrm{and} \quad 0< h_\rho < 2 ( 2-\rho  ) \quad \forall \rho \in \left ( 0,2 \right ).
$$
Secondly, we define
$$
m_\rho :=\rho^{-1}\Big ( 2\min\left \{ \rho ,\rho ^{-1} \right \}-\min\left \{ 1,\rho ^2 \right \} \Big)l_\rho,
$$
$$
n_\rho :=\rho^{-1}\left(2-\rho \right) \Big[ 2\left ( 2-\rho  \right )-h_\rho\Big],
$$
and
$$
o_\rho :=\left(2-\rho \right) \left( 2-\rho l_\rho\right) .
$$
Thirdly, we define
\begin{equation*}
	\mathcal{M_\rho }:=\begin{pmatrix}
		\mathcal{T}+\Sigma _g+\rho ^{-1}\sigma \mathcal{B}\mathcal{B}^* & \rho ^{-1}\left ( 1-\rho  \right )\mathcal{B}\\
		\rho ^{-1}\left ( 1-\rho  \right )\mathcal{B}^* & \left ( \sigma \rho  \right )^{-1}\mathcal{I}
	\end{pmatrix},
\end{equation*}
and
\begin{equation}\label{eq:29}
	\mathcal{M}:=\begin{pmatrix}
		\mathcal{S}+\Sigma _f & 0\\
		0 & \mathcal{M_\rho }
	\end{pmatrix}+\frac{1}{4}o_\rho\sigma \varepsilon \varepsilon ^*,
\end{equation}
and
\begin{equation}\label{eq:30}
	\mathcal{H}:=
	\begin{pmatrix}
		\S+\Sigma _f &  0& 0\\
		0 & \T+\Sigma _g+\frac{1}{2}n_\rho \sigma \mathcal{B}\mathcal{B}^*&0 \\
		0&  0& \left ( 2\sigma \right )^{-1}m_\rho\mathcal{I}
	\end{pmatrix}+\frac{1}{8}o_\rho\sigma \varepsilon \varepsilon ^*,
\end{equation}
where $\I$ is an identity operator within an appropriate space, $\mathcal{\varepsilon }:\mathcal{X}\rightarrow \mathcal{U}:=\mathcal{Y\times Z\times X}$ is a linear operator such that its adjoint $\mathcal{\varepsilon }^*$ satisfies
\begin{equation}\label{epstar}
	\mathcal{\varepsilon }^*\left ( y,z,x \right )=\mathcal{A}^*y+\mathcal{B}^*z,\quad\forall \left ( y,z,x \right )\in \mathcal{Y\times Z\times X}.
\end{equation}
Finally, we define
$$
k_1:=3\left \| S \right \|,\; \; \; k_2:=\max\left \{ \left \| \T \right \|,\rho ^{-2}\left ( 3\lambda _{\max}\left ( \mathcal{A}\mathcal{A}^* \right )\sigma +2\left ( 1- \rho \right )^2\sigma ^{-1} \right )\right \},
$$
and
$$
k_3:=\frac{3}{2}\left ( 1-\rho  \right )^2\rho ^{-2}\sigma\lambda _{\max}\left ( \mathcal{A}\mathcal{A}^* \right )+ \sigma^{-1} \rho^{-2}.
$$
Then, we let $k_4:=\max\left \{ k_1,k_2,k_3 \right \}$.
Let  $\mathcal{H}_0$ be a block-diagonal linear operator defined by\[\mathcal{H}_0:=k_4\begin{pmatrix}
	\S&0& 0\\
	0 & \T+\sigma \mathcal{B}\mathcal{B}^*  & 0\\
	0&  0& \left( 2\sigma \right) ^{-1}\mathcal{I}
\end{pmatrix}.\]
Because $\rho \in \left ( 0,2 \right )$, and $\Sigma _f+\S+\sigma \mathcal{A}\mathcal{A}^*  $ and $ \Sigma _g+\T+\sigma \mathcal{B}\mathcal{B}^* $ are assumed to be positive definite, it is easy to see that   $\mathcal{M_\rho }$ is positive definite, $\mathcal{M}$,  $\mathcal{H}_0$ and $\mathcal{H}$ are positive semidefinite.

We now give a lemma to estimate the upper bound of the sequence $\{R(u^k)\}$.
\begin{lemma}\label{lemma3}
	Let $\left \{ u^k \right \}$ be the sequence generated by sPGADMM scheme (\eqref{eq:a}-\eqref{eq:c}). Then, for any $k\geq0$, we have
	\begin{equation}\label{eq:23}
		\left \| u^{k+1}-u^k \right \|_{\mathcal{H}_0}^2\geq \left \| R(u^{k+1}) \right \|^2.
	\end{equation}
\end{lemma}
\begin{proof}
	It follows from $(\ref{eq:3})$ and (\ref{eq:12}), we obtain
	\begin{equation}\label{eq:21}
		\begin{cases}
			y^{k+1}=\textrm{Prox}_f\Big ( y^{k+1}+\mathcal{A} [ x^{k+1} -\left ( 1-\rho  \right )\sigma \left ( \mathcal{A}^*y_e^{k+1}+\mathcal{B}^*z_e^{k+1} \right ) \Big. \\[3mm]
			\qquad\qquad \qquad \left. \Big.+\sigma \left ( 2-\rho  \right )\mathcal{B}^*\left ( z^{k+1}-z^k \right )\right ] -\S\left ( y^{k+1}-y^k \right )\Big ),\\[3mm]
			z^{k+1}=\textrm{Prox}_g\Big( z^{k+1}+\mathcal{B}x^{k+1} -\T\left ( z^{k+1}-z^k \right )\Big).
		\end{cases}
	\end{equation}
Then, combining (\ref{eq:21}) and (\ref{eq:c}), and noticing the Lipschitz continuity property of Moreau-Yosida proximal mapping, we can get from the definition of $R(\cdot)$ (\ref{eq:22}) that
\begin{equation*}
	\begin{split}
		\left \| R\left ( u^{k+1} \right ) \right \|^2
		\leq &\Big \|\S\left ( y^{k+1}-y^k \right ) + \left ( 1-\rho  \right )\rho ^{-1}\mathcal{A} \left ( x^k-x^{k+1} \right )+\rho ^{-1}\sigma \mathcal{A}\mathcal{B}^*\left ( z^k-z^{k+1} \right )\Big \|^2\\[3mm]
		&+\Big\| \T\left ( z^{k+1}-z^k \right ) \Big \|^2+\Big \|\left ( \sigma \rho  \right )^{-1}\left (  x^k-x^{k+1} \right )+\rho ^{-1}\left ( 1-\rho  \right )\mathcal{B}^*\left ( z^k-z^{k+1} \right )\Big\|^2\\[3mm]
		\leq & 3\left \| \S \right \|\left \| y^{k+1}-y^k \right \|_\S^2+ 3\left ( 1-\rho  \right )^2\rho ^{-2}\lambda _{\max}\left ( \mathcal{A}\mathcal{A}^* \right )\left \| x^k-x^{k+1} \right \|^2\\[3mm]
&+3\rho ^{-2}\sigma ^2\lambda _{\max}\left ( \mathcal{A}\mathcal{A}^* \right )\left \|\mathcal{B}^*\left (  z^{k+1}-z^k \right ) \right \|^2
		+\left \| \T \right \|\left \| z^{k+1}-z^k \right \|_\T^2\\[3mm]
&+2\left ( \sigma \rho  \right )^{-2}\left \| x^k-x^{k+1} \right \|^2+2\rho ^{-2}\left ( 1-\rho  \right )^2\left \| \mathcal{B}^*\left ( z^k-z^{k+1} \right ) \right \|^2\\[3mm]
		\leq& k_1\left \| y^{k+1}-y^k \right \|_\S^2+k_2\left \| z^{k+1}-z^k \right \|_{\sigma \mathcal{B}\mathcal{B}^*+\T}^2+k_3\left \| x^k-x ^{k+1} \right \|^2.
	\end{split}
\end{equation*}
Recalling the definition of $\mathcal{H}_0$, one can readily see that the inequality (\ref{eq:23}) holds. This completes the proof.
\end{proof}

In the following analysis, the $\mathcal{M}$ defined in (\ref{eq:29}) is used to measure the distance from an iterate to the KKT point, while the $\mathcal{H}$ given in (\ref{eq:30}) is used to serve the distance between two consecutive iterates.
The positive definiteness of the linear operators $\M$ and $\H$ is proved at the following lemma.
\begin{lemma}\label{prop3}
	Let $\rho \in \left ( 0,2 \right )$. Then, it holds that
	$$
 \Big(\Sigma _f+\S+\sigma \mathcal{A}\mathcal{A}^*\succ 0 \  \textrm{and} \  \Sigma _g+\T+\sigma \mathcal{B}\mathcal{B}^*\succ 0  \Big)\Longleftrightarrow \mathcal{M}\succ 0\Longleftrightarrow \mathcal{H}\succ 0,
	$$
	where the symbol `$\succ 0$' means that the associated operator is positive definite.
\end{lemma}
\begin{proof}
	We only prove the first equivalence because the second one can be got in a similar way.
	Suppose on the contrary that there exists a non-zero vector $d:=\left ( d_y,d_z,d_x \right )\in \mathcal{Y\times Z\times X}$ such that $\left \langle d,\mathcal{M}d \right \rangle=0$ in the cases of $\Sigma _f+\S+\sigma \mathcal{A}\mathcal{A}^*\succ 0 $ and $\Sigma _g+\T+\sigma \mathcal{B}\mathcal{B}^*\succ 0$.
    Using the definition of $\mathcal{M}$ in (\ref{eq:29}), it means that
    $$
    \langle d_y, ( S+\Sigma _f ) d_y \rangle=0,\qquad \langle  ( d_z,d_x ),\mathcal{M}_\rho  ( d_z,d_x )  \rangle=0,\quad\textrm{and}\quad\mathcal{A}^*d_y+\mathcal{B}^*d_z=0,
    $$
    which, combines with the fact that $\mathcal{M}_\rho\succ 0$ and that $\Sigma _f+S+\sigma \mathcal{A}\mathcal{A}^*\succ 0 $,  it yields that $d=0$. This result contradicts to the assumption $d\neq 0$, and hence, $\mathcal{M}\succ 0$.
    Conversely, suppose that $\mathcal{M}\succ 0$. Because $o_\rho >0$ is a positive scalar, for any $d:=\left ( d_y,0,0 \right )\in \mathcal{Y\times Z\times X}$, it holds that $\left \langle d_y,\left ( \S+\Sigma _f +\frac{1}{4}o_\rho \sigma \mathcal{A}\mathcal{A}^*\right ) d_y\right \rangle>0$, which means $\Sigma _f+S+\sigma \mathcal{A}\mathcal{A}^*\succ 0$. In a similar way, for any $d:=\left ( 0,d_z,0 \right )\in \mathcal{Y\times Z\times X}$, we get $\left \langle d_z,\left [ \T+\Sigma _g +\left( \rho ^{-1}+\frac{1}{4}o_\rho\right ) \sigma \mathcal{B}\mathcal{B}^*\right ] d_z\right \rangle> 0$, which means that $\Sigma _g+\T+\sigma \mathcal{B}\mathcal{B}^*\succ 0 $. In a summary, the first equivalence is proved.
\end{proof}

Based on Lemmas \ref{lemma3}  and \ref{prop3}, we are ready to present a key inequality required for the coming linear convergence rate analysis.
\begin{lemma}
	Let $\rho \in (0,2)$ and $\left \{  u^k \right \}$ be a sequence generated by sPGADMM scheme (\eqref{eq:a}-\eqref{eq:c}). Then, for any $\bar{u}=\left ( \bar{y},\bar{z},\bar{x} \right )\in \bar{\Omega }$ and any $k\geq 0$, we have
	\begin{equation}\label{eq:27}
		\begin{split}
			\left \| u^{k} -\bar{u}\right \|_\mathcal{M}^2&+\left ( 2-\rho  \right )\rho ^{-1}\left \| z^k-z^{k-1} \right \|_\T^2\\[2mm]
			\geq &\left(\left \| u^{k+1} -\bar{u}\right \|_\mathcal{M}^2+\left ( 2-\rho  \right )\rho ^{-1}\left \| z^{k+1}-z^k \right \|_\T^2\right)+\left \| u^{k+1}-u^k \right \|_\mathcal{H}^2.
		\end{split}
	\end{equation}
Consequently, it also holds for $k\geq 0$ that
\begin{equation}\label{eq:28}
	\begin{split}
		\textrm{dist}_\mathcal{M}^2\left ( u^{k},\bar{\Omega } \right )&+\left ( 2-\rho  \right )\rho ^{-1}\left \| z^k-z^{k-1} \right \|_\T^2\\[2mm]
		\geq &\left (\textrm{dist}_\mathcal{M}^2\left ( u^{k+1},\bar{\Omega } \right )+\left ( 2-\rho  \right )\rho ^{-1}\left \| z^{k+1}-z^k \right \|_\T^2  \right )+\left \| u^{k+1}-u^k \right \|_\mathcal{H}^2.
	\end{split}
\end{equation}
\end{lemma}
\begin{proof}
	From (\ref{eq:c}), it holds that
	\begin{equation*}
		\begin{split}
			&\left \| x^{k+1}-x^k\right \|^2
			=\left ( \sigma \rho  \right )^2\left \| \mathcal{A}^*y_e^{k+1}+\mathcal{B}^*z_e^{k+1} \right \|^2-\sigma ^2\left ( 1-\rho  \right )^2\left \| \mathcal{B}^*\left ( z^{k+1}-z^k \right ) \right \|^2\\[2mm]
			&+2\left ( 1-\rho  \right )\left \langle x^k-x^{k+1},\sigma \mathcal{B}^* \left ( z^{k+1}-z^k \right )\right \rangle\\[2mm]
			&\leq \left ( \sigma \rho  \right )^2\left \| \mathcal{A}^*y_e^{k+1}+\mathcal{B}^*z_e^{k+1} \right \|^2-\sigma ^2\left ( 1-\rho  \right )^2\left \| \mathcal{B}^*\left ( z^{k+1}-z^k \right ) \right \|^2\\[2mm]
			&+\left ( 1-\min\left \{ \rho ,\rho ^{-1} \right \} \right )\left \| x^{k+1}-x^k\right \|^2+\rho \left ( \max\left \{ \rho ,\rho ^{-1} \right \} -1\right )\sigma ^2\left \| \mathcal{B}^*\left ( z^{k+1}-z^k \right ) \right \|^2.
		\end{split}
	\end{equation*}
Let $\bar{u}=\left ( \bar{y},\bar{z},\bar{x} \right )\in \bar{\Omega }$. Using Lemma \ref{lemma2} and above inequality, it yields that
	\begin{equation*}
		\begin{split}
			\left ( \sigma \rho  \right )^{-1} &\left \| x_e^{k}+\sigma\left( 1-\rho\right) \mathcal{B}^*z_e^k\right \|^2 +\left \|  y_e^k\right \|_S^2+\left \|  z_e^k\right \|_T^2+\sigma \left ( 2-\rho  \right )\left \|\mathcal{B}^*z_e^k \right \|^2\\[2mm]
 &+\rho ^{-1}\left ( 2-\rho  \right )\left \| z^k- z^{k-1}  \right \|_\T^2\\[2mm]
			\geq &\left ( \sigma \rho  \right )^{-1}\left \| x_e^{k+1}+\sigma\left( 1-\rho\right) \mathcal{B}^*z_e^{k+1}\right \|^2 +\left \|  y_e^{k+1}\right \|_\S^2+\left \|  z_e^{k+1}\right \|_\T^2\\[2mm]
&+\sigma \left ( 2-\rho  \right )\left \|\mathcal{B}^*z_e^{k+1}   \right \|^2+\rho ^{-1}\left ( 2-\rho  \right )\left \| z^{k+1}- z^k  \right \|_\T^2\\[2mm]
&+\left ( 2\sigma \rho  \right )^{-1}\left ( 2\min\left \{ \rho ,\rho ^{-1} \right \}-\min\left \{ 1,\rho ^2 \right \} \right )l_\rho \left \| x^{k+1}-x^k \right \|^2\\[2mm]
&+\left \| y^{k+1}-y^k \right \|_\S^2+\left \| z^{k+1}-z^k \right \|_\T^2+\left ( 2-\rho  \right )^2\rho ^{-1}\sigma \left \|\mathcal{B}^*\left ( z^{k+1}-z^k   \right )\right \|^2
		\end{split}
	\end{equation*}

	\begin{equation*}
		\begin{split}
&+2\left \| y_e^{k+1} \right \|_{\Sigma _f}^2+2\left \| z_e^{k+1} \right \|_{\Sigma _g}^2
			+\left ( 2-\rho  \right )\sigma  \left \| \mathcal{\varepsilon }^*\left ( y^{k+1},z^{k+1},0 \right )-c  \right \|^2\\[2mm]
&-\frac{1}{2}\left ( 2-\rho  \right )\rho \sigma l_\rho \left \| \varepsilon ^*\left ( y^{k+1},z^{k+1},x^{k+1} \right )-c \right \|^2\\[2mm]
			&-\frac{1}{2}\left ( 2-\rho  \right )\left ( 1-\min\left \{ \rho ,\rho ^{-1} \right \} \right )\sigma l_\rho \left \| \mathcal{B}^*\left ( z^{k+1} -z^k\right ) \right \|^2\\[2mm]
			=&\left ( \sigma \rho  \right )^{-1}\left \| x_e^{k+1}+\sigma\left( 1-\rho\right) \mathcal{B}^*z_e^{k+1}\right \|^2 +\left \|  y_e^{k+1}\right \|_\S^2+\left \|  z_e^{k+1}\right \|_\T^2+\sigma \left ( 2-\rho  \right )\left \|\mathcal{B}^*z_e^{k+1}   \right \|^2 \\[2mm]
			&+\rho ^{-1}\left ( 2-\rho  \right )\left \| z^{k+1}- z^k  \right \|_T^2+\left ( 2\sigma  \right )^{-1}m_\rho\left \| x^{k+1}-x^k \right \|^2\\[2mm]
			&+\left \| y^{k+1}-y^k \right \|_S^2+\left \| z^{k+1}-z^k \right \|_T^2+\frac{1}{2}n_\rho\sigma \left \|\mathcal{B}^*\left ( z^{k+1}-z^k   \right )\right \|^2\\[2mm]
			&+\frac{1}{2}o_\rho\sigma  \left \| \mathcal{\varepsilon }^*\left ( y^{k+1},z^{k+1},x^{k+1} \right )-c  \right \|^2+2\left \| y_e^{k+1} \right \|_{\Sigma _f}^2+2\left \| z_e^{k+1} \right \|_{\Sigma _g}^2,
		\end{split}
	\end{equation*}
\noindent which is equivalent to
\begin{equation}\label{eq:24}
	\begin{split}
		\left ( \sigma \rho  \right )^{-1}&\left \| x_e^{k}+\sigma\left( 1-\rho\right) \mathcal{B}^*z_e^k\right \|^2+\left \|  y_e^k\right \|_S^2+\left \|  z_e^k\right \|_T^2+\sigma \left ( 2-\rho  \right )\left \|\mathcal{B}^*z_e^k \right \|^2 \\[2mm]
		&+\rho ^{-1}\left ( 2-\rho  \right )\left \| z^k- z^{k-1}  \right \|_\T^2+\left \| y_e^{k} \right \|_{\Sigma _f}^2+\left \| z_e^{k} \right \|_{\Sigma _g}^2+\frac{1}{4}o_\rho\sigma \left \| \mathcal{\varepsilon }^*\left ( y^{k},z^{k},x^k \right )-c  \right \|^2\\[2mm]
		\geq& \left[ \left ( \sigma \rho  \right )^{-1}\left \| x_e^{k+1}+\sigma\left( 1-\rho\right) \mathcal{B}^*z_e^{k+1}\right \|^2 +\left \|  y_e^{k+1}\right \|_\S^2+\left \|  z_e^{k+1}\right \|_T^2+\sigma \left ( 2-\rho  \right )\left \|\mathcal{B}^*z_e^{k+1}   \right \|^2 \right.\\[2mm]
		&+\rho ^{-1}\left ( 2-\rho  \right )\left \| z^{k+1}- z^k  \right \|_\T^2 +\left \| y_e^{k+1} \right \|_{\Sigma _f}^2+\left \| z_e^{k+1} \right \|_{\Sigma _g}^2\\[2mm]
		&\left.+\frac{1}{4}o_\rho\sigma  \left \| \mathcal{\varepsilon }^*\left ( y^{k+1},z^{k+1},x^{k+1} \right )-c  \right \|^2\right] +\left ( 2\sigma  \right )^{-1}m_\rho\left \| x^{k+1}-x^k \right \|^2+\left \| y^{k+1}-y^k \right \|_\S^2 \\[2mm]
		&+\left \| z^{k+1}-z^k \right \|_T^2+\frac{1}{2}n_\rho\sigma \left \|\mathcal{B}^*\left ( z^{k+1}-z^k   \right )\right \|^2+\left \| y_e^{k} \right \|_{\Sigma _f}^2+\left \| y_e^{k+1} \right \|_{\Sigma _f}^2+\left \| z_e^{k} \right \|_{\Sigma _g}^2\\[2mm]
&+\left \| z_e^{k+1} \right \|_{\Sigma _g}^2
		+\frac{1}{4}o_\rho\sigma \left \| \mathcal{\varepsilon }^*\left ( y^{k+1},z^{k+1},x^{k+1} \right )-c  \right \|^2+\frac{1}{4}o_\rho\sigma \left \| \mathcal{\varepsilon }^*\left ( y^{k},z^{k},x^k \right )-c  \right \|^2.
	\end{split}
\end{equation}
Moreover, from (\ref{epstar}) we have
\begin{equation*}
	\begin{aligned}
		\mathcal{\varepsilon }^*\left ( y^{k+1},z^{k+1},x^{k+1} \right )-c=\mathcal{A}^*\left ( y^{k+1}-\bar{y} \right )+\mathcal{B}^*\left ( z^{k+1} -\bar{z}\right ),\\[2mm]
		\mathcal{\varepsilon }^*\left ( y^{k+1},z^{k+1},x^{k+1} \right )-c=\mathcal{A}^*\left ( y^{k+1}-\bar{y} \right )+\mathcal{B}^*\left ( z^{k+1} -\bar{z}\right ),
	\end{aligned}
\end{equation*}
and we also have
\begin{equation*}
	\begin{aligned}
		&\left \| y^{k+1}-\bar{y} \right \|_{\Sigma _f}^2+\left \| y^{k}-\bar{y} \right \|_{\Sigma _f}^2\geq \frac{1}{2}\left \| y^{k+1}-y^k\right \|_{\Sigma _f}^2,\\[2mm]
		&\left \| z^{k+1}-\bar{z} \right \|_{\Sigma _g}^2+\left \| z^{k}-\bar{z} \right \|_{\Sigma _g}^2\geq \frac{1}{2}\left \| z^{k+1}-z^k\right \|_{\Sigma _g}^2,\\[2mm]
		&\left \| \mathcal{\varepsilon }^*\left ( y^{k+1},z^{k+1},x^{k+1} \right )-c \right \|^2+\left \|  \mathcal{\varepsilon }^*\left ( y^{k},z^{k},x^k \right )-c\right \|^2\geq \frac{1}{2}\left \| \mathcal{\varepsilon }^*\left ( y^{k+1}-y^k,z^{k+1} -z^k,x^{k+1}-x^k \right ) \right \|^2.
	\end{aligned}
\end{equation*}
Based on the above relations, one can see from (\ref{eq:24})  for any $\rho\in(0,2)$ that
\begin{equation}
	\begin{split}
		\left ( \sigma \rho  \right )^{-1}&\left \| x_e^{k}+\sigma\left( 1-\rho\right) \mathcal{B}^*z_e^k\right \|^2+\left \|  y_e^k\right \|_\S^2+\left \|  z_e^k\right \|_\T^2+\sigma \left ( 2-\rho  \right )\left \|\mathcal{B}^*z_e^k \right \|^2 \\[2mm]
		&+\rho ^{-1}\left ( 2-\rho  \right )\left \| z^k- z^{k-1}  \right \|_\T^2+\left \| y_e^{k} \right \|_{\Sigma _f}^2+\left \| z_e^{k} \right \|_{\Sigma _g}^2+\frac{1}{4}o_\rho\sigma \left \| \mathcal{\varepsilon }^*\left ( y_e^{k},z_e^{k},x_e^k \right ) \right \|^2\\[2mm]
		\geq& \left[ \left ( \sigma \rho  \right )^{-1}\left \| x_e^{k+1}+\sigma\left( 1-\rho\right) \mathcal{B}^*z_e^{k+1}\right \|^2 +\left \|  y_e^{k+1}\right \|_\S^2+\left \|  z_e^{k+1}\right \|_\T^2+\sigma \left ( 2-\rho  \right )\left \|\mathcal{B}^*z_e^{k+1}   \right \|^2 \right.\\[2mm]
		&+\rho ^{-1}\left ( 2-\rho  \right )\left \| z^{k+1}- z^k  \right \|_\T^2 +\left \| y_e^{k+1} \right \|_{\Sigma _f}^2+\left \| z_e^{k+1} \right \|_{\Sigma _g}^2\\[2mm]
		&\left.+\frac{1}{4}o_\rho\sigma  \left \| \mathcal{\varepsilon }^*\left ( y_e^{k+1},z_e^{k+1},x_e^{k+1} \right ) \right \|^2\right]+\left ( 2\sigma \right )^{-1}m_\rho \left \| x^{k+1}-x^k \right \|^2 \\[2mm]
		&+\left \| y^{k+1}-y^k \right \|_S^2+\left \| z^{k+1}-z^k \right \|_\T^2+\frac{1}{2}n_\rho\sigma \left \|\mathcal{B}^*\left ( z^{k+1}-z^k   \right )\right \|^2\\[2mm]
		&+\frac{1}{2}\left \| y^{k+1}-y^k \right \|_{\Sigma _f}^2+\frac{1}{2}\left \| z^{k+1}-z^k \right \|_{\Sigma _g}^2\\[2mm]
&+\frac{1}{8}o_\rho\sigma \left \|  \mathcal{\varepsilon }^*\left ( y^{k+1}-y^k,z^{k+1} -z^k,x^{k+1}-x^k \right )  \right \|^2,
	\end{split}
\end{equation}
which implies the inequality (\ref{eq:27}) holds. Because $\bar{\Omega }$ is assumed to be nonempty and that (\ref{eq:27}) holds for any $\bar{u}\in\bar{\Omega }$, we get that desired result (\ref{eq:28})  holds.
\end{proof}

To prove the local linear convergence rate of sPGADMM, we needs the following calmness assumption on $R^{-1}$ at the original point for a KKT point.
\begin{assumption}\label{assumption2}
	The inverse mapping of KKT mapping $R(\cdot)$ defined in (\ref{eq:22}) is calm at the origin for $\bar{u}$ with modulus $\kappa >0$ in the sense that there exists a constant $r>0$ such that
	\begin{equation*}
		\textrm{dist}\left ( u,\bar{\Omega } \right )\leq \kappa \left \| R\left ( u \right ) \right \|,\qquad\forall u\in \left \{ u\in \mathcal{U}|\left \| u-\bar{u} \right \| \leq r\right \}.
	\end{equation*}
\end{assumption}

Now, we are ready to present the local linear convergence rate of  sPGADMM.
\begin{theorem}\label{theorem2}
	Suppose that Assumptions \ref{assumption1} and \ref{assumption2} hold. Assume that $\Sigma _f+\S+\sigma \mathcal{A}\mathcal{A}^*\succ 0$ and $\Sigma _g+\T+\sigma \mathcal{B}\mathcal{B}^*\succ 0 $.
	Then the sequence  $\left\{(y^k,z^k,x^k)\right\}$ generated by sPGADMM converges to a KKT point $\bar{u}=\left ( \bar{y},\bar{z},\bar{x} \right )\in \bar{\Omega }$, and there exists a  $\bar{k}\geq 0$ such that for all $k\geq \bar{k}$ it holds that
	\begin{equation}\label{eq:35}
		\textrm{dist}_{\bar{\mathcal{M}}}^2\left (u^{k+1}  ,\bar{\Omega }\right )+\left \| z^{k+1}- z^{k}  \right \|_\T^2\leq  \vartheta\left [\textrm{dist}_{\bar{\mathcal{M}}}^2\left (u^{k}  ,\bar{\Omega }\right )+\left \| z^{k}- z^{k-1}  \right \|_\T^2  \right ],
	\end{equation}
	where $\bar{\mathcal{M}}:=\rho \left ( 2-\rho  \right )^{-1}\mathcal{M}
	$ and $\vartheta:=\left [ 1+k_5\rho\left(\rho +\left ( 2-\rho  \right )k_5 \right )^{-1} \right ]^{-1}$
	with
	$$
	k_5:=\min\left \{ 1,m_\rho ,\frac{1}{2}n_{\rho} \right \}k_4^{-1}\rho \left ( 2-\rho  \right )^{-1}\kappa ^{-2}\lambda _{\max}\left ( \bar{\mathcal{M}} \right )^{-1},
	$$
	where $\lambda_{\max}(\cdot)$ denotes the largest eigenvalue of a given matrix.
	Moreover, there also exists a positive number $\eta \in \left [ \vartheta ,1 \right ) $ such that for all $k\geq \bar{k}$,
	\begin{equation}\label{eq:36}
		\textrm{dist}_{\bar{\mathcal{M}}}^2\left (u^{k+1}  ,\bar{\Omega }\right )+\left \| z^{k+1}- z^{k}  \right \|_\T^2\leq  \eta\left [\textrm{dist}_{\bar{\mathcal{M}}}^2\left (u^{k}  ,\bar{\Omega }\right )+\left \| z^{k}- z^{k-1}  \right \|_\T^2  \right ].
	\end{equation}
\end{theorem}
\begin{proof}
	From Theorem \ref{theorem1}, we know that the sequence $\left\{(y^k,z^k,x^k)\right\}$ generated by  sPGADMM converges to a KKT point $\bar{u}=(\bar{y},\bar{z},\bar{x})\in\bar{\Omega }$, which shows that there exist $\bar{k}\geq 0$ and $r>0$ such that for any $k\geq \bar{k}$,
	\[\left \| u^{k+1}-\bar{u} \right \|\leq r.\]
	Subsequently, from Assumption \ref{assumption2} and Lemma \ref{lemma3}, it gets for all $k\geq \bar{k}$ that
	\begin{equation}\label{eq:31}
		\textrm{dist}^2\left ( u^{k+1},\bar{\Omega } \right )\leq \kappa ^2\left \| R\left ( u^{k+1} \right ) \right \|^2\leq \kappa ^2\left \| u^{k+1}-u^k \right \|_{\mathcal{H}_0}^2.
	\end{equation}
Denoting $\bar{\mathcal{H}}:=\rho \left ( 2-\rho  \right )^{-1}\mathcal{H}$ and noticing
$\bar{\mathcal{M}}:=\rho \left ( 2-\rho  \right )^{-1}\mathcal{M}
$, then (\ref{eq:28}) can be simplified as
\begin{equation}\label{eq:34}
	\textrm{dist}_{\bar{\mathcal{M}}}^2\left ( u^{k},\bar{\Omega } \right )+\left \| z^k-z^{k-1} \right \|_\T^2\geq \left (\textrm{dist}_{\bar{\mathcal{M}}}^2\left ( u^{k+1},\bar{\Omega } \right )+\left \| z^{k+1}-z^k \right \|_\T^2  \right )+\left \| u^{k+1}-u^k \right \|_{\bar{\mathcal{H}}}^2.
\end{equation}
    Recalling the structure of $\H$, we know for all $k\geq 0$ that
    \begin{equation}\label{eq:39}
    	\left \| u^{k+1}-u^k \right \|_{\bar{\mathcal{H}}}^2\geq \rho \left ( 2-\rho  \right )^{-1}\left \| z^{k+1}-z^k \right \|_\T^2.
    \end{equation}
    Observing the structures of $\bar{\mathcal{H}}$ and $\mathcal{H}_0$, it holds that
    \begin{equation}\label{eq:32}
    	k_4\bar{\mathcal{H}}\geq \min\Big \{ 1,m_\rho ,\frac{1}{2}n_{\rho} \Big \}\rho \left ( 2-\rho  \right )^{-1}\mathcal{H}_0+\frac{1}{8}k_4\rho \left ( 2-\rho  \right )^{-1}o_\rho \sigma \varepsilon \varepsilon ^*,\qquad \forall \rho \in \left ( 0,2 \right ).
    \end{equation}
    Combining (\ref{eq:31}) and (\ref{eq:32}), we get for $k\geq \bar{k}$ that
    \begin{equation}\label{eq:33}
    	\begin{split}
    		\left \| u^{k+1}-u^k \right \|_{\bar{\mathcal{H}}}^2&\geq \min\Big \{ 1,m_\rho ,\frac{1}{2}n_{\rho} \Big \}k_4^{-1}\rho \left ( 2-\rho  \right )^{-1}\left \| u^{k+1}-u^k \right \|_{\mathcal{H}_0}^2\\[2mm]
    		&\geq \min\Big \{ 1,m_\rho ,\frac{1}{2}n_{\rho} \Big \}k_4^{-1}\rho \left ( 2-\rho  \right )^{-1}\kappa ^{-2}\textrm{dist}^2\left ( u^{k+1},\bar{\Omega } \right )\\[2mm]
    		&\geq k_5\textrm{dist}_{\bar{\mathcal{M}}}^2\left ( u^{k+1},\bar{\Omega } \right ).
    	\end{split}
    \end{equation}
Let $k_6:=\rho\left(\rho +\left ( 2-\rho  \right )k_5 \right )^{-1}$. Accordingly,  for all $k\geq \bar{k}$, it follows from (\ref{eq:33}), (\ref{eq:39})and (\ref{eq:34}) that
\begin{equation*}
	\begin{split}
	\textrm{dist}_{\bar{\mathcal{M}}}^2&\left ( u^{k},\bar{\Omega } \right )+\left \| z^k-z^{k-1} \right \|_\T^2\\[2mm]
	\geq &\textrm{dist}_{\bar{\mathcal{M}}}^2\left ( u^{k+1},\bar{\Omega } \right )+\left \| z^{k+1}-z^k \right \|_\T^2+k_6\left \| u^{k+1}-u^k \right \|_{\bar{\mathcal{H}}}^2+\left (1-k_6\right )\left \| u^{k+1}-u^k \right \|_{\bar{\mathcal{H}}}^2\\[2mm]
	\geq&\left ( 1+ k_5k_6 \right )\textrm{dist}_{\bar{\mathcal{M}}}^2\left ( u^{k+1},\bar{\Omega } \right )+\left [ 1+ \left (1-k_6\right )\rho \left ( 2-\rho  \right )^{-1}\right ]\left \| z^{k+1}-z^k \right \|_\T^2.
	\end{split}
\end{equation*}
From the fact  that $1+ k_5k_6= 1+ \left (1-k_6\right )\rho \left ( 2-\rho  \right )^{-1}$, we get the assertion (\ref{eq:35}) holds with $ \vartheta=\left( 1+ k_5k_6\right) ^{-1}$.
Because $ \vartheta<1$,  there must exists a positive number $\eta \in \left [ \vartheta ,1 \right ) $ such that  (\ref{eq:36}) holds. This completes the proof.
\end{proof}

From Theorem \ref{theorem2}, we know that the calmness assumption on $R^{-1}$ is needed to ensure the local linear convergence rate of sPGADMM.
From the preliminaries reviewed in Subsection \ref{sec2.2} we know that, $R^{-1}$ is calm if it is a piecewise polyhedral multivalued mapping, and particularly, it is the subdifferential mapping of a convex piecewise linear-quadratic function. Using the fact that $R^{-1}$ is piecewise polyhedral if and only if $R$ itself is piecewise polyhedral, we conclude that the calmness property holds automatically for convex composite optimization problem (\ref{eq:1}), as well as the following corollary.
\begin{corollary}\label{corof}
	Let $\rho\in\left( 0,2\right) $. Suppose that the set $\bar{\Omega }$ is nonempty and  that both $\Sigma _f+S+\sigma \mathcal{A} \mathcal{A}^*$ and $\Sigma _g+T+\sigma \mathcal{B} \mathcal{B}^*$ are positive definite. Let $u^k=\left\{(y^k,z^k,x^k)\right\}$ be the sequence generated by Algorithm sPGADMM that converges to $\bar{u}\in\bar{\Omega }$. Assume that the mapping $R(\cdot)$ is piecewise polyhedral. Then, there exists a constant $\hat{\kappa }>0$ such that for all $k\geq 0$ ,
	\begin{equation}\label{eq:37}
		\textrm{dist}\left ( u^k,\bar{\Omega } \right )\leq \hat{\kappa } \| R\left ( u^k \right ) \|,
	\end{equation}
and
\begin{equation}\label{eq:38}
	\textrm{dist}_{\bar{\mathcal{M}}}^2\left (u^{k+1}  ,\bar{\Omega }\right )+\left \| z^{k+1}- z^{k}  \right \|_\T^2\leq  \hat{\vartheta}\left [\textrm{dist}_{\bar{\mathcal{M}}}^2\left (u^{k}  ,\bar{\Omega }\right )+\left \| z^{k}- z^{k-1}  \right \|_\T^2  \right ],
\end{equation}
	where
	\[\hat{\vartheta}=\left [ 1+\hat{k_5}\rho\left(\rho +\left ( 2-\rho  \right )\hat{k_5} \right )^{-1} \right ]^{-1} \quad \textrm{and} \quad \hat{k_5}=\min\left \{ 1,m_\rho ,\frac{1}{2}n_{\rho} \right \}k_4^{-1}\rho \left ( 2-\rho  \right )^{-1}\hat{\kappa} ^{-2}\lambda _{\max}\left ( \bar{\mathcal{M}} \right )^{-1}.\]
\end{corollary}
\begin{proof}
	Because $\bar{\Omega }$ is nonempty and $R^{-1}$ is piecewise polyhedral, we know that there exist constants $\kappa>0$ and $\zeta>0$ such that
	\[\textrm{dist}\left ( u,\bar{\Omega } \right )\leq \kappa \left \| R\left ( u \right ) \right \|,\qquad \forall u\in \left \{ u\in\mathcal{U} \ | \ \left \| R\left ( u \right ) \right \| \leq \zeta \right \}.\]
	Besides, according to the proof of Theorem \ref{theorem2}, there exists a constant $r>0$ such that for all $k\geq 1$, the sequence $\left\{(y^k,z^k,x^k)\right\}$ generated by sPGADMM converges to a KKT point $\bar{u}=\left ( \bar{y},\bar{z},\bar{x} \right )$ if $\left \| u^k-\bar{u} \right \|\leq r$. With these $u^k$  satisfying $\left \| R\left ( u^k \right ) \right \|>\zeta $, we have
	$$
	\textrm{dist}\left ( u^k,\bar{\Omega } \right )\leq \left \| u^k-\bar{u} \right \|\leq r< r\big( \zeta ^{-1} \left \| R\left ( u^k \right ) \right \|\big ).
	$$
	Then (\ref{eq:37}) holds immediately with $\hat{\kappa }:=\max\left \{ \kappa ,r\zeta ^{-1} \right \}$ for all $k\geq1$. The inequality (\ref{eq:38}) can be proved in a similar way to the one in Theorem \ref{theorem2}.
\end{proof}

To end this section, we note that the Corollary \ref{corof} is
the global linear convergence rate of sPGADMM, and for convex composite quadratic programming, it is holds with no additional conditions because the error bound assumption holds automatically.

\section{Conclusions}\label{sec5}

The generalized ADMM for convex composite optimization problems was firstly proposed by Eckstein \& Bertsekas \cite{Eckstein1992OnTD} in $1992$, and has been widely analyzed and implemented over the fast few decades. The convergence of this method can be obtained from the framework of PPA, but the proof from the aspect of optimization is relatively fewer. This paper considered a sPGADMM which covers almost all the existed generalized ADMMs as special cases, but importantly, the subproblems would be solved more efficiently with the help of the semi-proximal term. This paper focused on proving the linear convergence rate of sPGADMM under the condition that the mapping $R^{-1}(\cdot)$ is calmness at the original point for $\bar u$ with a suitable modulus, because this condition is very mild and holds automatically for the considered problem (\ref{eq:1}). The required condition is same to the work of Han et al. \cite{penglinear} in establishing the linear convergence rate of the semi-proximal ADMM of Fazel et al. \cite{Fazel2013HankelMR}. To some extent, our paper can be considered as an extension of the work of Han et al. \cite{penglinear} to the generalized variant of ADMM. However, we must note that, this extension is not trivial, because the relaxation factor has caused many obstacles for convergence analysis.
Certainly, this is the motivation as well as the contribution of this paper.


\section*{Disclosure statement}
The authors report there are no competing interests to declare.

\section*{Funding}
This work was supported by the National Natural Science Foundation of China under Grant [number 11971149].

\bibliography{references}  
\bibliographystyle{plain}

\end{document}